\documentclass[11pt,english]{amsart}
\usepackage{dsfont,amssymb,amsmath,mlmodern,mathrsfs,fourier,baskervald}

\usepackage[text={6.5in,8.5in},centering]{geometry}

\usepackage{comment}
\usepackage[dvipsnames]{xcolor}   
\usepackage{xparse}
\usepackage{xr-hyper}
\usepackage[linktocpage=true,colorlinks=true,hyperindex,citecolor=Red,linkcolor=blue]{hyperref}  

\usepackage[all]{xy}
\usepackage{tikz-cd}
\newdir{ >}{!/6pt/@{ }*:(1,0.2)@_{>}*:(1,-0.2)@^{>}}
\usepackage{mathtools}

\newcommand{\Max}{\operatorname{Max}}

\newtheorem{theorem}{Theorem}[section]
\newtheorem{proposition}[theorem]{Proposition}
\newtheorem{lemma}[theorem]{Lemma}
\newtheorem{corollary}[theorem]{Corollary}
\theoremstyle{definition}
\newtheorem{definition}{Definition}[section]

\numberwithin{equation}{section}
\theoremstyle{remark}
\newtheorem{remark}[theorem]{Remark}

\begin{document}
	
\title{Strong Hollowness in Commutative Rings}
	
\author{Amartya Goswami}
\address{[1]  Department of Mathematics and Applied Mathematics, University of Johannesburg, P.O. Box 524, Auckland Park 2006, South Africa. [2]  National Institute for Theoretical and Computational Sciences (NITheCS), South Africa.}
\email{agoswami@uj.ac.za}

\author{Joseph Israel Zelezniak}

\address{Department of Mathematics and Applied Mathematics, University of Cape Town, Rondebosch 7700, South Africa}
\email{ZLZJOS001@myuct.ac.za}
	
\subjclass{13A15, 13C13, 13F99}
%Ideals and multiplicative ideal theory in commutative rings

%Other special types of modules and ideals in commutative rings

%Arithmetic rings and other specialcommutative rings: None of the above, but in this section

\keywords{Strongly hollow ideal, Jacobson radical,  discrete valuation ring}
	
\begin{abstract}
In this paper we study strongly hollow ideals and completely strongly hollow ideals in commutative rings without finiteness
assumptions. We establish basic structural properties, including maximality phenomena and permanence under quotients and
surjective homomorphisms. We obtain several characterizations of completely strongly hollow ideals in terms of extremal ideals
avoiding a given ideal, and we show that a strongly hollow ideal which is not contained in the Jacobson radical is necessarily
completely strongly hollow. As applications, we derive strong restrictions in integral domains and consequences for principal
ideal domains, including a discrete valuation ring criterion. We develop the connection between complete hollowness and
complete irreducibility and obtain a correspondence between completely strongly hollow ideals and completely strongly irreducible
ideals. Finally, we develop a condition related to greatest common divisors which is equivalent to strongly hollowness under mild finiteness conditions.
\end{abstract}

\maketitle 
%\tableofcontents

\section{Introduction}
Let $R$ be a commutative ring with identity. Following Rostami \cite{Rostami2021}, an ideal $I$ of $R$ is called
\emph{strongly hollow} if for all ideals $A,B$ of $R$,
\[
I\subseteq A+B \quad\Longrightarrow\quad I\subseteq A \ \text{or}\ I\subseteq B,
\]
and \emph{completely strongly hollow} if for every family of ideals $\{J_\omega\}_{\omega\in\Omega}$ of $R$,
\[
I\subseteq \sum_{\omega\in\Omega}J_\omega \quad\Longrightarrow\quad I\subseteq J_{\omega_0},
\ \text{for some}\ \omega_0\in\Omega.
\]
These “hollowness” conditions capture a rigidity phenomenon for containment in sums and were initiated, in the present form in \cite{Rostami2021}.
Subsequent work has developed further variants and applications, including connections with topological constructions and additional finiteness hypotheses; see, for instance,
\cite{Ceken-Yuksel} and the references therein.

A recurring theme in the study of hollowness is that the behaviour of the family of ideals \emph{not}
containing $I$ controls both structural and localization-theoretic properties of $I$.
Accordingly, for an ideal $I$ we consider the ideal
\[
\Gamma_I \ :=\ \sum_{\,I\not\subseteq J} J,
\]
and, when convenient, the associated colon ideal
\[
L_I \ :=\ \Gamma_I : I.
\]
For strongly hollow ideals, these objects measure how far $I$ is from being forced into a given ideal:
they also provide a bridge to the theory of strongly irreducible and completely irreducible ideals
(\textit{cf}.\ \cite{FHO,HRR2002}).

The first part of Section~\ref{GR} develops general existence properties for strongly hollow ideals.
We show that inside any fixed ideal $I$ there is either no strongly hollow ideal at all, or there is an inclusion-maximal one
(Theorem~\ref{max_sh}).
We then study how hollowness behaves under standard constructions: colon ideals with respect to non-zero divisors
(Proposition~\ref{colon_ideal}), passage to quotients including explicit formulae for $\Gamma$ and $L$ which refine previous characterizations,
and images/preimages under surjective homomorphisms, where a “small kernel” hypothesis recovers hollowness in the source.

A central goal of the rest of Section~\ref{GR} is to clarify the relationship between strongly hollow and completely strongly hollow ideals
without imposing finiteness assumptions. This allows us to extend many notions developed in \cite{Rostami2021}, which were previously only considered for finitely generated ideals.
We give two complementary characterizations of completely strongly hollow ideals:
(i) in terms of being the least ideal (by inclusion) that fails containment in a suitable ideal (Theorem~\ref{least}),
and (ii) in terms of the existence of a greatest ideal avoiding $I$ (Theorem~\ref{Gamma}).
These characterizations sharpen and extend the ring results of \cite{Rostami2021} and interact cleanly with arithmetical hypotheses.
We also isolate a particularly effective criterion in terms of the Jacobson radical:
if a strongly hollow ideal is not contained in $J(R)$, then it is automatically completely strongly hollow and detects a unique maximal ideal
(Theorem~\ref{escapes_jacobian}).

The end of Section~\ref{GR} shows that strong hollowness imposes severe restrictions in classical settings.
For integral domains, we prove that every proper strongly hollow ideal is contained in $J(R)$ (Theorem~\ref{int_domain_radical}),
yielding, in particular, that a semi-primitive domain admits a strongly hollow ideal only in the field case (Corollary~\ref{cont}).
In the principal ideal domain setting this culminates in a sharp DVR criterion:
a PID admits a proper strongly hollow ideal if and only if it is a discrete valuation ring (Proposition~\ref{PID_to_DVR}).

We then relate strong hollowness to the depth of maximal ideals.
If $I\subseteq J(R)$ is strongly hollow, then at every maximal $M$ with $\dim_{R/M}(M/M^2)\ge 2$ one has $I\subseteq M^2$
(Theorem~\ref{I_in_M2}),
and more generally, there is at most one maximal ideal for which $I$ fails containment in some power $M^n$
(Proposition~\ref{comax_powers}).

In Section~\ref{3} we make precise the duality between complete hollowness and complete irreducibility.
In the local case, completely strongly hollow ideals correspond bijectively to strongly irreducible ideals $J$ with $J\neq (J:M)$
(Theorem~\ref{bijection}), and we extend this to an intrinsic bijection between completely strongly hollow ideals and completely strongly
irreducible ideals in arbitrary rings (Theorem~\ref{true_bij}), in the spirit of \cite{FHO,HRR2002}.
In the Noetherian and Artinian settings, we extract further consequences and reformulations (including $P$-primary strongly irreducible ideals and
finite-intersection phenomena).
Section~\ref{4} introduces a GCD-ring condition ($\star$) ensuring strong hollowness under mild hypotheses.

\smallskip
\section{General Results}
\label{GR}

\begin{theorem}\label{max_sh}
Let $R$ be a ring and $I$ be an ideal of $R$. Then exactly one of the following holds:
\begin{enumerate}
\item[\textup{(a)}] There is no strongly hollow ideal contained in $I$.
\item[\textup{(b)}] There exists an (inclusion-)maximal strongly hollow ideal $H$ with $H\subseteq I$.
\end{enumerate}
\end{theorem}

\begin{proof}
Let
\[
S\;:=\;\{\,K \text{ an ideal of } R\mid\ K\subseteq I\ \text{and }K\text{ is strongly hollow}\,\}.
\]
If $S=\varnothing$, we are in case \textup{(a)}, and there is nothing to prove.
Assume $S\neq\varnothing$. To apply Zorn's lemma, let $\widetilde S\subseteq S$ be a non-empty chain
(with respect to inclusion) and set
\[
H\;:=\;\bigcup_{K\in\widetilde S} K.
\]
Because $\widetilde S$ is a chain, $H$ is an ideal, and clearly $H\subseteq I$.
We claim that $H$ is strongly hollow. Suppose $H\subseteq A+B$ for ideals $A$, $B$ of $R$.
Assume, for contradiction, that $H\nsubseteq A$ and $H\nsubseteq B$.
Pick $x\in H\setminus A$ and $y\in H\setminus B$, so $x\in K_x$ and $y\in K_y$, for some $K_x$, $K_y\in\widetilde S$.
Since $\widetilde S$ is a chain, either $K_x\subseteq K_y$ or $K_y\subseteq K_x$.
Without loss of generality, suppose $K_x\subseteq K_y$. Hence $x$, $y\in K_y$ and $K_y\subseteq A+B$.
Because $K_y$ is strongly hollow, $K_y\subseteq A$ or $K_y\subseteq B$.
If $K_y\subseteq A$, then $x\in A$ giving a contradiction; if $K_y\subseteq B$, then $y\in B$, which gives a contradiction.
Thus, one of $H\subseteq A$ or $H\subseteq B$ must hold, proving that $H$ is strongly hollow.
Therefore, $H\in S$ and $H$ is an upper bound in $S$ for the chain $\widetilde S$.
By Zorn's lemma, $S$ has a maximal element $H_{\max}$ with $H_{\max}\subseteq I$.
This gives case \textup{(b)}.
\end{proof}

\begin{proposition}
    Let $I$ be a completely strongly hollow ideal of a ring $R$. Then $I\cap \Gamma_I$ is the greatest ideal strictly contained in $I$ and $I+\Gamma_I$ is the least ideal strictly containing $\Gamma_I$. 
\end{proposition}

\begin{proof}
    Let $J$ be an ideal of $R$. If $J\subset I$, then $I\not\subseteq J$, and thus $J\subseteq \Gamma_I$, which gives that $J\subseteq I\cap\Gamma_I$. The same argument gives the rest of the claim. 
\end{proof}

\begin{proposition}\label{colon_ideal}
    Let $I$ be an ideal of $R$ and $a\in R$ such that $a$ is a non-zero divisor. If $Ra\cap I$ is strongly hollow (completely strongly hollow) then $I:Ra$ is strongly hollow (completely strongly hollow).
\end{proposition}
\begin{proof}
    Suppose that $(I:Ra)\subseteq \sum_{\omega\in\Omega}J_\omega$, for some family of ideals $\{J_\omega \}_{\omega\in\Omega}$. Then \[Rb\cap Ra=(I:Ra)a\subseteq \sum_{\omega\in\Omega}J_\omega a,\] where the first equality is shown in \cite[Lemma.~2.5]{HRR2002} for principal ideals but holds for all ideals $I$. By assumption, there exists $\omega\in\Omega$ such that $I\cap Ra\subseteq J_\alpha a$. Taking the colon ideal with respect to $Ra$ on both sides again gives $(I\cap Ra:Ra)\subseteq (J_\omega a:Ra)$. This is equivalent to $(I:Ra)\subseteq (J_\omega a:a)=J_\omega +(0:a)=J_\omega$.
\end{proof}

\begin{proposition}
Let $I$, $J$ be ideals of $R$ such that $I\not\subseteq J$. Then $(I+J)/J$ is a strongly hollow ideal in $R/J$ if $I$ is a strongly hollow ideal in $R$. In this case, \[\Gamma_{(I+J)/J}=\Gamma_I/J\quad \text{and}\quad L_{(I+J)/J}=L_I/J.\]
\end{proposition}
\begin{proof}
Let $S$, $T$ be two ideals of $R$ such that $S$, $T\supseteq J$ and $(I+J)/J\subseteq ((S/J)+(T/J))$. Then, we have that $(I+J)/J\subseteq (S+T)/J$, and hence, $I\subseteq I+J\subseteq S+T$. As $I$ is a strongly hollow ideal of $R$, we have that $I\subseteq S$ or $I\subseteq T$. This in turn gives that $I+J\subseteq S+J=S$ or $I+J\subseteq T+J=T$. We conclude that $(I+J)/J\subseteq S/J$ or $(I+J)/J\subseteq T/J$. 

Now suppose $I\not\subseteq S$. By way of obtaining a contradiction, suppose that $I+J\subset S+J$. Since $I\subseteq I+J\subseteq S+J$ and $I$ is a strongly hollow ideal, we must have that $I\subseteq J$ or $I\subseteq S$, which is a contradiction. Therefore, $I+J\not\subseteq S+J$ and $(I+J)/J\not\subseteq (S+J)/J$. Then $(S+J)/J\subseteq \Gamma_{(I+J)/J} $ and since $S$ was an arbitrary ideal not containing $I$, we have that \[\sum_{I\not\subseteq S}(S+J)/J=\left(\sum_{I\not\subseteq S}S+J\right)/J=(\Gamma_I+J)/J.\] But since $I\not\subseteq J$, we have that $J\subseteq \Gamma_I$. Therefore $\Gamma_I/J\subseteq \Gamma_{(I+J)/J}$. Now let $S/J$ be an ideal in $R/J$ such that $(I+J)/J\not\subseteq S/J$. Then we have that $I+J\not\subset S$, and hence, $I\not\subseteq S$. By definition, we have that $S\subseteq \Gamma_I$, and therefore, $S/J\subseteq\Gamma_I/J$. Since $S/J$ was an arbitrary ideal such that $S/J\not\supseteq (I+J)/J$, we have that $\Gamma_{(I+J)/J}\subseteq\Gamma_I/J$. We conclude that $\Gamma_{(I+J)/J}=\Gamma_I/J$. For the colon ideal we can calculate it directly. $$L_{(I+J)/J}=(\Gamma_{(I+J)/J}):\left((I+J)/J\right)=(\Gamma_I/J):((I+J)/J)=(\Gamma_I:I+J)/J.$$
Finally, we have \[\Gamma_I:(I+J=(\Gamma_I:I)\cap(\Gamma_I:J)=\Gamma_I:I=L_I.\] We conclude that $L_{(I+J)/J}=L_I/J$. 
\end{proof}
This generalizes the results \cite[Proposition.~3.21]{Ceken-Yuksel} and \cite[Proposition.~1.19]{Rostami2021} from completely strongly hollow ideals ( respectively, principally generated strongly hollow ideals ) to arbitrary strongly hollow ideals. The result additionally does away with an unnecessary factor in the colon ideal. The study of arbitrary strongly hollow ideals in comparison to those that are principally generated is significantly underdeveloped.

In fact we can show a more general and stronger result that surjective ring homomorphisms preserve strongly hollow ideals. 
\begin{proposition}
Let $\varphi:R\to R'$ be a surjective ring homomorphism and let $I$ be a strongly hollow ideal of $R$ such that $I\not\subseteq \ker\varphi$. Then $\varphi(I)$ is strongly hollow in $R'$. The same result holds when one replaces strongly hollow with completely strongly hollow.
\end{proposition}

\begin{proof}
Suppose $\varphi(I)\subseteq J'+K'$ for ideals $J'$, $K'$ of $ R'$. Taking preimages, we have
\[
  I \subseteq \varphi^{-1}(J'+K')=\varphi^{-1}(J')+\varphi^{-1}(K').
\]
Since $I$ is strongly hollow, either $I\subseteq \varphi^{-1}(J')$ or $I\subseteq \varphi^{-1}(K')$. Hence $\varphi(I)\subseteq J'$ or $\varphi(I)\subseteq K'$. Since the surjective image of a finitely generated ideal is also finitely generated, we have that the same result holds for completely strongly hollow ideals.  
\end{proof}

The following result can be seen as dual to the fact that surjective ring homomorphisms preserve prime ideals which contain the kernel. It shows that the preimage of a strongly hollow ideal under a surjective ring homomorphism is strongly hollow in the sublattice of ideals containing the kernel.

\begin{proposition}\label{ker-preimage}
Let $\varphi:R\to R'$ be surjective and let $J'$ be a strongly hollow ideal of $ R'$. Then, for all ideals $A$, $B$ of $R$,
\[
  \varphi^{-1}(J')\subseteq A+B \ \Rightarrow\ \varphi^{-1}(J')\subseteq A+\ker\varphi \ \ \text{or}\ \ \varphi^{-1}(J')\subseteq B+\ker\varphi.
\]
\end{proposition}

\begin{proof}
From $\varphi^{-1}(J')\subseteq A+B$ and surjectivity, we get
$J'\subseteq \varphi(A)+\varphi(B)$. Since $J'$ is strongly hollow, $J'\subseteq \varphi(A)$ or $J'\subseteq \varphi(B)$. Taking preimages gives $H\subseteq A+\ker\varphi$ or $H\subseteq B+\ker\varphi$.
\end{proof}

\begin{definition}
A submodule $N$ of a module $M$ is \emph{small} (written $N\ll M$) if $N+L=M$ implies $L=M$ for every submodule $L\subseteq M$.
\end{definition}

\begin{theorem}
Let $\varphi:R\to R'$ be surjective, $J'$ strongly hollow in $R'$, and $H:=\varphi^{-1}(J')$. If $\ker\varphi\ll H$, then $H$ is strongly hollow in $R$. The same result holds when replacing strongly hollow with completely strongly hollow.
\end{theorem}

\begin{proof}
Assume $H\subseteq A+B$. By Proposition~\ref{ker-preimage}, $H\subseteq A+\ker\varphi$ or $H\subseteq B+\ker\varphi$. Without a loss of generality, assume $H\subseteq A+\ker\varphi$. Intersect with $H$ and use the modular law (since $\ker\varphi\subseteq H$):
\[
  H=(A+\ker\varphi)\cap H=(A\cap H)+\ker\varphi.
\]
Because $\ker\varphi\ll H$, we must have $A\cap H=H$. Therefore, \ $H\subseteq A$. The other case is symmetric.
\end{proof}

\begin{theorem}\label{least}
    Let $R$ be a ring and $I$ an ideal of $R$. Then $I$ is a completely strongly hollow ideal if and only if there exists an ideal $K$ such that $I$ is the least (by inclusion) ideal with respect to not being contained by $K$.
\end{theorem}
\begin{proof}
    Suppose $I$ is completely strongly hollow. Let $J$ be an ideal of $R$ such that $J\not\subseteq \Gamma_I$. By the definition of $\Gamma_I$ we have that $I\subseteq J$. Therefore, $I$ is the least ideal with respect to not being contained in $\Gamma_I$. Conversely, suppose there exists $K$ such that $I$ is the least ideal with respect to not being contained in $K$. Let $S$, $T$ be two ideals such that $I\subseteq S+T$. If $I\not\subseteq S$ and $I\not\subseteq T$, then $S+T\subseteq K$ giving a contradiction. Therefore $I$ is strongly hollow. Let $x\in I$, if $x\not\in K$, then $I=(x)$ since it is the least ideal not contained in $K$. If for all $x\in I$ we have that $x\in K$, then $I\subseteq K$. Therefore, there exists $x\in R$ such that $I=Rx$. We conclude that $I$ is completely strongly hollow.
\end{proof}

We now generalize \cite[Theorem.~1.14]{Rostami2021} by showing that one can remove the assumption that $I$ is finitely generated. 
\begin{theorem}\label{Gamma}
    Let $I$ be a non-zero ideal of a ring $R$. Then $I$ is completely strongly hollow if and only if there exists the greatest ideal of $R$ with respect to not containing $I$.
\end{theorem}
\begin{proof}
    The forward is shown in \cite[Theorem.~1.14]{Rostami2021}. Now suppose there exists a greatest ideal $J$ of $R$ with respect to not containing $I$. Then $I\not\subseteq J$ by definition. Suppose there exists a family of ideals $\{K_\alpha\}_{\alpha \in \Omega}$ such that $I\subseteq \sum_{\alpha \in \Omega}K_\alpha$ with $I\not\subseteq K_\alpha$ for all $\alpha \in \Omega$. Then we have that $K_\alpha \subseteq J$, for all $\alpha \in \Omega$. Therefore $\sum_{\alpha\in \Omega}K_\alpha \subseteq J$, which is a contradiction. Therefore $I\subseteq K_\alpha$, for some $\alpha \in \Omega$. Furthermore, in this case we have that the greatest ideal with respect to not containing $I$ is $\Gamma_I$. 
\end{proof}

\begin{remark}
This provides the elegant characterization of completely strongly hollow ideals in the following statement: An ideal $I$ of a ring $R$ is completely strongly hollow if and only if $I\not\subseteq \Gamma_I$. The first characterization can have `least' reduced to `minimal' when considering arithmetical rings. 	
\end{remark}

\begin{proposition}
    Let $R$ be an arithmetical ring and $I$ an ideal of $R$. Then $I$ is a completely strongly hollow ideal if and only if there exists an ideal $K$ such that $I$ is a minimal (by inclusion) ideal with respect to not being contained by $K$.
\end{proposition}

\begin{proof}
    The forward direction remains unchanged. Suppose there exists an ideal $K$ such that $I$ is minimal with respect to not being contained in $K$. Let $S$, $T$ be two ideals of $R$ such that $I\subseteq S+T$. Suppose that $I\not\subseteq S$ and $I\not\subseteq T$. Then $I\cap S\subseteq K$ and $I\cap T\subseteq K$ by minimality of $I$. Hence $I\cap(S+T)\subseteq K$ giving $I\subseteq K$, a contradiction. Therefore, $I$ is strongly hollow. One can easily see that $I$ is a principal ideal by the same argument as in Theorem~\ref{least}. We conclude that $I$ is completely strongly hollow. 
\end{proof}

We can replace the requirement that $R$ be arithmetic with a tighter restriction on the ideals in question.	

\begin{corollary}
    Let $K$ be a strongly irreducible ideal of a ring $R$. If there exists an ideal $I$ minimal with respect to not being included in $K$, then $I$ is completely strongly hollow.
\end{corollary}

\begin{proof}
 Let $I$ be an ideal minimal with respect to not being contained in $K$. Then, following the proof above, we arrive at $I\cap S\subseteq K$ giving $S\subseteq K$. Then we must have $T\not\subseteq K$, and hence, $I\cap T=I$. Therefore $I\subseteq T$.
\end{proof}
The previous two results also give rise to dual results for strongly irreducible ideals. We state the dual results below the proof of which is a straightforward exercise in redoing the previous proofs in the dual lattice of ideals. 
\begin{corollary}
    Let $R$ be an arithmetical ring and $K$ and ideal of $R$. Then $K$ is a strongly irreducible ideal if and only if there exists an ideal $I$ such that $K$ is a maximal (by inclusion) ideal with respect to not containing  $I$.
\end{corollary} 
\begin{corollary}
      Let $I$ be a strongly hollow ideal of a ring $R$. If there exists an ideal $K$ maximal with respect to not containing $I$, then $K$ is strongly irreducible.
\end{corollary}

The following results give conditions on when a strongly hollow ideal is necessarily completely strongly hollow. One can also consider Theorem \ref{escapes_jacobian} to be an extension of one part of \cite[Theorem.~1.27]{Rostami2021} since it shows that the results follow for arbitrary strongly hollow ideals.

\begin{theorem}\label{escapes_jacobian}
    Let $I$ be a strongly hollow ideal of a ring $R$. Then $I$ is a completely strongly hollow ideal if $I\not\subseteq J(R)$. In this case, $\Gamma_I=M$, where $M$ is the unique maximal ideal of $R$ not containing $I$.
\end{theorem}
\begin{proof}
    Suppose $I\not\subseteq J(R)$. Then by \cite[Proposition.~3.8]{Ceken-Yuksel}, there exists a unique maximal ideal $M$ of $R$ such that $I\not\subseteq M$. Therefore $M\subseteq\Gamma_I$, and hence, $\Gamma_I=M$ or $\Gamma_I=R$. Suppose $\Gamma_I=R$. Then there exists an ideal $J$ of $R$ such that $I\not\subseteq J$ and $J\not\subseteq M$. Considering the sum $J+M=R\supseteq I$, we find that $I\subseteq J$ or $I\subseteq M$. Therefore $\Gamma_I=M$, which is a contradiction. 
\end{proof}

\begin{remark}
Theorem~\ref{escapes_jacobian} gives a weaker condition for all the results from \cite[Theoremm.~1.27]{Rostami2021}. It also shows that both \cite[Proposition.~2.14]{Rostami2021} and \cite[Proposition.~2.15]{Rostami2021} hold for arbitrary strongly hollow ideals rather than just completely strongly hollow ideals.
\end{remark}

Before we state the next result, we need to recall a definition. A \emph{semi-primitive} ring is a ring which has zero Jacobson radical.

\begin{corollary}
    Every strongly hollow ideal in a semi-primitive ring is completely strongly hollow. 
\end{corollary}

\begin{proposition}\label{si_min}
    Let $I$ be a strongly hollow ideal of a ring $R$. If $I\not\subseteq J(R)$, then $I$ is minimal with respect to not being contained within $J(R)$.
\end{proposition}

\begin{proof}
    Let $J$ be an ideal of $R$ such that $(0)\subset J\subseteq I$. If $J\not\subseteq \Gamma_I$, then $I\subseteq J$, and so $J=I$. Therefore, if $J\neq I$, we have that $J\subseteq J(R)$. We conclude that $I$ is minimal with respect to not being contained within $J(R)$.
\end{proof}

\begin{theorem}
    Let $I$ be an ideal of a ring $R$. If $I$ is minimal with respect to not being contained in $J(R)$, then $I$ is completely strongly hollow.
\end{theorem}

\begin{proof}
    By Theorem~\ref{escapes_jacobian}, we  only need to show that $I$ is strongly hollow. There must exists a maximal ideal $M$ such that $I\not\subseteq M$ since $I\not\subseteq J(R)$. Therefore, by the minimality of $I$, we have $I\cap M\subseteq J(R)$. Since all maximal ideals are strongly irreducible, we can conclude that $I\subseteq N$, for all $N\in\mathrm{Max}(R)\textbackslash\{M\}$. Then $I$ is minimal with respect to not being contained in $M$. We conclude by Proposition \ref{si_min} that $I$ is strongly hollow. 
\end{proof}

\begin{proposition}
    Let $I\not\subseteq J(R)$ be an ideal of a ring $R$ and $a\in R$ a non-zero divisor. If $I\cap Ra$ is strongly hollow, then $I=(I:Ra)$ and $I$ is a completely strongly hollow ideal. In this case, we have that $Ia=I\cap Ra$
\end{proposition}

\begin{proof}
    We have from Proposition~\ref{colon_ideal} that $(I:Ra)$ is strongly hollow. It is clear that $I\subseteq (I:Ra)$, and thus $(I:Ra)\not\subseteq J(R)$. Then $(I:Ra)$ is a completely strongly hollow ideal by Theorem~\ref{escapes_jacobian} and is minimal with respect to not being contained in the Jacobian radical by Proposition~\ref{si_min}. This forces $I=(I:Ra)$, and thus $Ia=(I:Ra)a=I\cap Ra$.
\end{proof}

\begin{theorem}\label{int_domain_radical}
    Let $R$ be an integral domain and $I$ a proper ideal of $R$. If $I$ is strongly hollow, then $I\subseteq J(R)$. 
\end{theorem}

\begin{proof}
    Suppose, by way of obtaining a contradiction that $I$ is strongly hollow and $I\not\subseteq J(R)$. Then, by Theorem \ref{escapes_jacobian}, $I$ is completely strongly hollow. Therefore $I=Ra$, for some $a\in R$. Since $R$ is an integral domain, $a$ is a non-zero-divisor. Therefore $R$ is a local ring by \cite[Proposition.~1.9]{Rostami2021}. By assumption $I\not\subseteq J(R)=M$, where $M$ is the unique maximal ideal of $R$. Therefore $I=R$, contradicting the assumption that $I$ was a proper ideal of $R$. Therefore, $I$ must be contained in the Jacobian radical of $R$. 
\end{proof}

\begin{corollary}\label{cont}
    Let $R$ be a semi-primitive domain. Then there exists a strongly hollow ideal $I$ of $R$ if and only if $R$ is a field. 
\end{corollary}

\begin{proof}
    Suppose that $I\neq R$. Then by Proposition \ref{int_domain_radical} we have that $I\subseteq J(R)=\{0\}$, which is a contradiction since strongly hollow ideals are non-zero ideals. Thus $I=R$. Therefore $I\not\subseteq J(R)$, and hence, is completely strongly hollow. Therefore $R$ is local by \cite[Proposition.~1.12]{Rostami2021}. Then $R$ is a local semi-primitive domain which has its unique maximal ideal being the zero ideal. Therefore $R$ is a field. Conversely if $R$ is a field, then $R$ is trivially a strongly hollow ideal.
\end{proof}

\begin{remark}
The above result shows that for any integral domain $R$ the polynomial ring $R[x]$ has no strongly hollow ideals.	
\end{remark}

\begin{proposition}
        Let $R$ be a ring and $I$ and ideal of $R$. If $\Gamma_I$ is a proper ideal, then either $I\subseteq J(R)$ or $I$ is completely strongly hollow.
\end{proposition}

\begin{proof}
    Suppose that $I$ is not completely strongly hollow, and hence, $I\subseteq \Gamma_I$. If $I\not\subseteq J(R)$, then there exists a maximal ideal $M$ such that $I\not\subseteq M$. Thus $M\subseteq \Gamma_I$, which gives that $M=\Gamma_I$, and hence $I\subseteq M$, which is a contradiction. Thus, if $I\subseteq \Gamma_I$, then $I\subseteq J(R)$. Now suppose $I\not\subseteq\Gamma_I$. Then $I$ is completely strongly hollow.
\end{proof}

Before we go for the next proposition, we introduce a terminology. We call a pair of ideals $J$, $K$ of a ring $R$ \emph{weakly coprime} if $J+K=M$, for some maximal ideal $M$ of $R$.

\begin{proposition}
 Let $R$ be a ring and $I$ an ideal of $R$. If $\Gamma_I$ is a proper ideal, then $I$ is strongly hollow with respect to the sum of coprime ideals. In the case, if $\Gamma_I$ is not a maximal ideal, then $I$ is strongly hollow with respect to the sum of weakly coprime ideals.
\end{proposition}

\begin{proof}
Suppose $I\subseteq J+K=R$. Then if $I\not\subseteq J$ and $I\not\subseteq K$, we would have that $R=J+K\subseteq \Gamma_I$, contradicting our assumption that $\Gamma_I$ is a proper ideal. The rest of the statement follows trivially in the case that $\Gamma_I$ is not a maximal ideal.
\end{proof}

\begin{proposition}\label{PID_to_DVR}
    Let $R$ be a PID. Then $R$ is a discrete valuation ring if and only if there exists a proper ideal $I$ which is strongly hollow.
\end{proposition}

\begin{proof}
   Suppose $R$ is a discrete valuation ring. Then the unique maximal ideal is completely strongly hollow by \cite[Proposition.~1.11]{Rostami2021}. Conversely, suppose there exists a completely strongly hollow ideal $I\neq R$. It is clear that $R$ is local. If $J(R)=(0)$, then Corollary~\ref{cont} gives that $R$ is a field, and hence $I=(0)$, which contradicts the assumption that $I$ is completely strongly hollow. Therefore $J(R)\neq(0)$. Any local PID which has non-zero Jacobian radical is either a field or a discrete valuation ring. We have shown that $R$ cannot be a field, and hence, it is a discrete valuation ring. 
\end{proof}

\begin{corollary}
    Let $R$ be a PID which is not a discrete valuation ring. Then there exists a strongly hollow ideal $I$ of $R$ if and only if  $R$ is local. In this case the unique strongly hollow ideal of $R$ is $R$ itself.
\end{corollary}

\begin{corollary}
    Let $R$ be a PID which is not a field. Then there exists a proper strongly hollow ideal $I$ of $R$ if and only if every ideal in $R$ is strongly hollow.
\end{corollary}

\begin{proof}
    The backwards implication is trivial. For the forward implication, notice that Proposition \ref{PID_to_DVR} implies  $R$ is a DVR, and hence, the ideals in $R$ form a chain. Therefore, all ideals in $R$ are comparable (by inclusion). By \cite[Theorem.~2.5]{Rostami2021} we conclude that every ideal of $R$ is strongly hollow.
\end{proof}

The following result extends \cite[Proposition~1.24]{Rostami2021} to strongly hollow ideals and gives a condition which is easier to check.

\begin{proposition}\label{ann_condition}
    Let $I$ be an ideal of a ring $R$ and $M\in \mathrm{Max}(R)$. If there exists an ideal $K$ such that $K\subseteq M$ and $K$ is not contained in any other maximal ideal with $IK=(0)$ and $I$ is strongly hollow in $R_M$, then $I$ is strongly hollow. 
\end{proposition}

\begin{proof}
    We have that $\mathrm{Ann}(I)\supseteq K$, and hence $\mathrm{Ann}(I)\not\subseteq N$, for any maximal ideal $N\in\mathrm{Max}(R)\textbackslash\{M\}$, and thus $\mathrm{Ann}(I)\subseteq M$. Let $N\in\mathrm{Max}(R)\textbackslash\{M\}$ and consider the localization \[(I\mathrm{Ann}(I))_N=I_N\mathrm{Ann}(I)_N=0_N.\] Since $\mathrm{Ann}(I)\not\subseteq N$, we have $I_NR_N=0_n$ giving $I_n=0_N$. Suppose that $I\subseteq S+T$. Then $I_N\subseteq S_N$ and $I_N\subseteq T_N$, and since $I_M$ is strongly hollow in $R_M$, we have that $I_M\subseteq S_M$ or $I_M\subseteq T_M$. Suppose, without loss of generality, that $I_M\subseteq S_M$. Then $I\subseteq S$ since the inclusion holds in the localization at every maximal ideal. Therefore, $I$ is strongly hollow.
\end{proof}

\begin{remark}
Notice that any $n$-th power of a maximal ideal $M$ or any $M$-primary ideal satisfies the conditions for $K$ above.	
\end{remark}

\begin{corollary}\label{s_ann_condition}
 Let $I$ be an ideal of a ring $R$ and $M$ a maximal ideal of $R$. Let $K$ be an ideal contained in $M$ and contained in no other maximal ideal such that $IK=(0)$. If $I\not\subseteq J(R)$, then $I$ is completely strongly hollow.
\end{corollary}

\begin{proof}
We only need to show that $I_M$ is completely strongly hollow in $R_M$. We will do this by showing that $I_M=R_M$, and then, use the fact that $R_M$ is completely strongly hollow since it is a local ring. Notice that $IK\subseteq N$, for all $N\in\mathrm{Max}(R)$. Hence $I\subseteq N$, for all $N\in\mathrm{Max}(R)\textbackslash\{M\}$. Since $I\not\subseteq J(R)$, we have that $I\not\subseteq M$ and we are done. 
\end{proof}

\begin{corollary}\label{strongest_ann}
    Let $I$ be a non-nilpotent ideal of a ring $R$ and $M$ a maximal ideal of $R$. If there exists an $M$-primary ideal $K$ such that $IK=(0)$, then $I$ is completely strongly hollow. Furthermore, in this case $I$ is contained in every prime ideal except for $M$ and $\mathrm{Ht}(M)=0$, i.e., $M$ is a minimal prime. Hence $R_M$ has zero Krull dimension.
\end{corollary}

\begin{proof}
    Let $P\in \mathrm{Spec}(R)$. If $P\neq M$, then $IK=(0)\subseteq P$ gives that $I\subseteq P$. Hence $I$ is contained in all prime ideals which are not $M$. Since $I\not\subseteq \mathrm{Nil}(R)$, it must be that $I\not\subseteq M$. Hence by Corollary \ref{s_ann_condition} $I$, is completely strongly hollow. Furthermore, since $I$ is contained in every prime ideal which is not equal to $M$ and $I\not\subseteq M$, we have that $\mathrm{Ht}(M)=0$. The rest of the statement follows by definition.
\end{proof}

\begin{remark}
The above corollary, once again, holds if $K$ is any power of a maximal ideal $M$. 
\end{remark}

\begin{corollary}
    Let $I$ be an ideal of a reduced Noetherian ring $R$ and $M$ a maximal ideal of $R$. If there exists an $M$-primary ideal $K$ which annihilates $I$, then $R_M$ is a field. 
\end{corollary}

\begin{proof}
    By Corollary \ref{strongest_ann}, we have that $M$ is a minimal prime. We conclude that $R_M$ is a field by \cite[\href{https://stacks.math.columbia.edu/tag/00EU}{Tag 00EU}]{Stacks}.
\end{proof}

To prove our next theorem, we need to recall the following well-known result.
\begin{lemma}\label{UH_avoid_W}
Let $V$ be a $k$-vector space with $\dim V\ge 2$, and let $0\neq W\subsetneq V$ be a proper subspace.
Then there exist subspaces $U,H\subset V$ such that
\[
\dim U=1,\qquad\dim H=\dim V-1,\qquad V=U\oplus H,
\]
and moreover $W\nsubseteq U$ and $W\nsubseteq H$.
In particular, if $\dim V=2$, both $U$ and $H$ may be chosen $1$-dimensional with $V=U\oplus H$.
\end{lemma}

\begin{proof}
Pick $u\in V\setminus W$ and set $U:=ku$, so $W\nsubseteq U$. 
Choose a linear functional $\varphi\in V^\ast$ with $\varphi(u)=1$ and $\varphi|_W\not\equiv 0$.
Let $H:=\ker\varphi$. Then $u\notin H$, so $V=U\oplus H$; and since $\varphi|_W\not\equiv 0$, there exists $w\in W$ with $\varphi(w)\neq 0$. Hence $w\notin H$ and $W\nsubseteq H$.
\end{proof}

\begin{theorem}\label{I_in_M2}
Let $R$ be a ring and $I\subseteq J(R)$ a strongly hollow ideal of $R$. Then for every maximal ideal $M$ with $\dim_{R/M}(M/M^2)\ge 2$ one has
$
I\ \subseteq\ M^2.
$
\end{theorem}

\begin{proof}
Fix a maximal ideal $M$, write $k:=R/M$ and $V:=M/M^2$. Since $I\subseteq J(R)\subseteq M$,
the image $W:=(I+M^2)/M^2\subseteq V$ is defined. Suppose $\dim_k V\ge 2$ and $I\nsubseteq M^2$. Then we have that $W\neq 0$. By Lemma~\ref{UH_avoid_W}, we can choose $U$, $H\subset V$ with
$\dim U=1$, $\dim H=\dim V -1$, $V=U\oplus H$, and $W\nsubseteq U$, $H$.
Let $\pi:M\to V$ be the quotient map and put
\[
J:=\pi^{-1}(U),\qquad K:=\pi^{-1}(H).
\]
Then
\[
\frac{J+K+M^2}{M^2}
=\frac{J+M^2}{M^2}+\frac{K+M^2}{M^2}
=U+H
=V
=\frac{M}{M^2},
\]
hence $J+K+M^2=M$. Localizing at $M$ and applying Nakayama’s Lemma in the local ring $R_M$ yields $(J+K)_M=M_M$. Hence $J+K=M$. Because $W\nsubseteq U$, $H$, we have
$(I+M^2)/M^2\nsubseteq(J+M^2)/M^2$ and $(I+M^2)/M^2\nsubseteq(K+M^2)/M^2$. Therefore, we must have $I\nsubseteq J$ and $I\nsubseteq K$. This contradicts the fact that $I$ is a strongly hollow ideal.
Therefore $W=0$, and we conclude that \ $I\subseteq M^2$.
\end{proof}

\begin{remark}
Notice that if $\dim_{R/M}(M/M^2)=0$, then $M=M^2$ and so the above result holds trivially. Therefore, the only case in which a strongly hollow ideal $I$ is such that $I\not\subseteq M^2$ for some maximal ideal $M$ is when $\dim_{R/M}(M/M^2)=1$. 	
\end{remark}

\begin{proposition}\label{comax_powers}
Let $R$ be a ring and let $I\subseteq R$ be a strongly hollow ideal of $R$. There exists at most one maximal ideal such that $I\not\subseteq M^n$, for some $n\in \mathds{N}$.
\end{proposition}

\begin{proof}
Suppose $M$ and $N$ are two distinct maximal ideals such that there exists $i$, $j\in\mathds{N}$ with $I\not\subseteq M^i$ and $I\not\subseteq N^j$. Since $M$ and $N$ are prime ideals, we have that $N^j\not\subseteq K$, for any maximal ideal $K$ which is not $N$ and $M^i\not\subseteq J$ for any maximal ideal $J$ which is not $M$. Therefore $M^i+N^j=R$, which contradicts the fact that $I$ is a strongly hollow ideal. Therefore $I$ is contained in either $M^n$ or $N^n$, for all $n\in\mathds{N}$. The uniqueness of the maximal ideal for which this does not hold follows by applying the same reasoning when assuming it is not unique.
\end{proof}

\begin{corollary}\label{A_contained_in_bad_M}
Let $R$ be a ring. Let $A$ and $I$ be ideals of $R$ with $I$ strongly hollow and $I\not\subseteq J(R)$. If $I\not\subseteq A^n$ for some $n\in\mathds{N}$, then $A\subseteq \Gamma_I$. 
\end{corollary}

\begin{proof}
$I$ is completely strongly hollow by Theorem~\ref{escapes_jacobian}. The rest of the statement follows due to the fact that $\Gamma_I$ is a prime ideal. 
\end{proof}

These results give a characterization of strongly hollow ideals in general as stated below.
 
\begin{theorem}
    Let $R$ be a ring and $I$ a strongly hollow ideal of $R$. Then one of the following holds:
    \begin{enumerate}
        \item There exists a unique maximal ideal $M$ for which $I\not\subseteq M$. In this case $I$ is a principal completely strongly hollow ideal with $\Gamma_I=M$ and, $$I\subseteq \bigcap_{\substack{n\in\mathds{N}\\ A\not\subseteq \Gamma_I}} A^n.$$ 
        \item $I$ is contained in all maximal ideals and all powers of all maximal ideals. That is, 
    $$I \subseteq \bigcap_{\substack{n\in\mathds{N}\\ M\in \Max{{R}}}}M^n.$$
    \item $I$ is contained in all maximal ideals and there exists a unique maximal ideal $M$ such that $I\not\subseteq M^n$, for some  $n\geq 2$.
    \end{enumerate}
\end{theorem}

\smallskip
\section{Completely Strongly Hollow Ideals and Completely Irreducible Ideals}
\label{3}
\begin{proposition}\label{ci_gives_ch}
    Let $(R,M)$ be a local ring. If $J$ is a strongly irreducible ideal of $R$ such that $J\subset (J:M)$, then $(J:M)$ is completely strongly hollow.
\end{proposition}
\begin{proof}
    By \cite[Theorem.~2.6(1)]{HRR2002}, we have that $(J:M)$ is a principal ideal. Choose $x\in R$ such that $(J:M)=Rx$. Let $S$, $T$ be ideals in $R$ such that $Rx\subseteq S+T$. Suppose, by way of obtaining a contradiction, that $Rx\not\subseteq S$ and $Rx\not\subseteq T$. Then by \cite[Theorem.~2.6(3)]{HRR2002}, we have that $Rx\subseteq S+T\subseteq J$, which contradicts our assumption that $J\subset (J:M)$. We conclude that $x$ is a strongly hollow element of $R$, and therefore $(J:M)$ is completely strongly hollow.
\end{proof}

\begin{remark}
In a local ring $(R,M)$, a strongly irreducible ideal $J$ which is strictly contained in the colon ideal $(J:M)$ is completely strongly irreducible (\cite[Theorem.~1.3(iv)]{FHO}).
\end{remark}

\begin{theorem}\label{bijection}
    Let $(R,M)$ be a local ring. Then there is a bijection between the set of completely strongly hollow ideals of $R$ and the set of strongly irreducible ideals, $J$, which are not equal to $J:M$. Equivalently, there exists a bijection between completely strongly hollow ideals and completely strongly irreducible ideals. 
\end{theorem}
\begin{proof}
    Let $I$ be a completely strongly hollow ideal. Then $\Gamma_I$ is a completely irreducible ideal. Therefore $\Gamma_I\subset(\Gamma_I:M)$ by \cite[Theorem.~1.3(iv)]{FHO}. We claim that $I=(\Gamma_I:M)$. It is clear that $I\subseteq(\Gamma_I:M)$ by the strict inclusion of $\Gamma_I$. Then using \cite[Theorem.~2.6(3)]{HRR2002} and $\Gamma_I\not\supseteq I$, we have we have that $(\Gamma_I:M)\subseteq  I$. We conclude that $(\Gamma_I:M)=I$. Now suppose $J$ is a strongly irreducible ideal such that $J\subset (J:M)$. Then we have that $(J:M)$ is completely strongly hollow by Proposition \ref{ci_gives_ch}. We claim that $\Gamma_{(J:M)}=J$. Since $J\not\supseteq (J:M)$, we have that $J\subseteq \Gamma_{(J:M)}$. We can again apply \cite[Theorem.~2.6(3)]{HRR2002} and the fact that $(J:M)\not\subseteq \Gamma_{(J:M)}$ to conclude that $\Gamma_{(J:M)}=J$. Therefore, the mapping which takes a completely strongly hollow ideal $I$ to $\Gamma_I$ and a completely strongly irreducible ideal $J$ to $(J:M)$ is a bijection.
\end{proof}
The bijection in Theorem~\ref{bijection} can be given in a more explicit form using \cite[Theorem.~2.6(2)]{HRR2002}. In a local ring $(R,M)$ every completely strongly irreducible ideal is of the form $\Gamma_I=IM$, for some completely strongly hollow ideal $I$. Every completely strongly hollow ideal is of the form $I=(\Gamma_I:M)=(IM:M)$ for some completely strongly irreducible ideal $\Gamma_I$.
The following result generalizes this connection and makes clear the duality between completely strongly irreducible and completely strongly hollow ideals.
\begin{theorem}\label{true_bij}
    In a ring, there exists a bijection between completely strongly hollow ideals and completely strongly irreducible ideals.
\end{theorem}

\begin{proof}
    Any completely strongly hollow ideal defines a completely strongly irreducible ideal via the greatest ideal with respect to not containing it. The other direction is  obvious. Given a completely strongly irreducible ideal $K$, consider the ideal $I=\bigcap_{J\not\subseteq I}J$. Note that $I$ is the least ideal with respect to not being contained in $K$, and thus, is completely strongly hollow by Proposition~\ref{least}. Finally, it is clear that $K=\Gamma_I$. 
\end{proof}

\begin{corollary}
    The bijection between completely strongly hollow and completely strongly irreducible ideals is an order preserving map. 
\end{corollary}
\begin{proof}
    This follows  from \cite[Corollary.~1.16]{Rostami2021}.
\end{proof}

Let us recall that a \emph{waist ideal} is an ideal which is comparable (under inclusion) to every other ideal of the ring. 

\begin{corollary}
    Let $(R:M)$ be a local ring. Every completely irreducible ideal of $R$ is a waist ideal.
\end{corollary}

\begin{proof}
    Theorem \ref{bijection} shows that in a local ring every completely strongly hollow ideal is of the form $\Gamma_I$ for a completely strongly hollow ideal $I$. The result now follows from \cite[Theorem.~2.5]{Rostami2021}.
\end{proof}

We now develop the relation between strongly hollow ideals and non-prime strongly irreducible ideals in Noetherian rings.
\begin{lemma}\label{cs_noeth}
    Every strongly hollow ideal in a Noetherian ring is completely strongly hollow.
\end{lemma}
\begin{proof}
    Every ideal in a Noetherian ring is finitely generated. A finitely generated strongly hollow ideal is completely strongly hollow. 
\end{proof}

\begin{theorem}
    Let $P$ be a prime ideal of a Noetherian ring $R$ and $I$ an ideal of $R$ such that $IR_P$ is a proper ideal of $R_P$. If $IR_P$ is strongly hollow, then:
    \begin{enumerate}
        \item $\Gamma_{IR_P}^{R_P}\cap R_P$ is a $P$-primary strongly irreducible ideal of $R$.
        
        \item $IR_P=\Gamma_{IR_p}^{R_P}:_{R_P}P$.
        
        \item For all ideals $J$ in $R$, either $J\subseteq \Gamma_{IR_P}^{R_P}\cap R_P$ or $IR_P\subseteq J$.
    \end{enumerate}
\end{theorem}

\begin{proof}
   For an ideal $J$ of $R$, denote $IR_P$ by $I^e$ and for an ideal $K$ of $R_P$, denote $K\cap R$ by $K^c$. \par (1) In $R_P$, note that $I^e$ is completely strongly hollow by Lemma \ref{cs_noeth}. Then $\Gamma_{I^e}^{R_P}$ is completely strongly irreducible in $R_P$. $I^e\not\subseteq \Gamma_{I^e}^{R_P}$ gives $\Gamma_{I^e}^{R_P}\neq P^e$. $\Gamma_{I^e}^{R_P}$ is $P^e$-primary in $R_P$ by \cite[Corollary.~1.5]{FHO}. Therefore by \cite[Lemma.~2.2(4)]{HRR2002} we have that $(\Gamma_{I^e}^{R_P})^c$ is strongly irreducible in $R$. Then since $R$ is Noetherian, $(\Gamma_{I^e}^{R_P})^c$ is a primary ideal. \par (2) $\Gamma_{I^e}^{R_P}$ is a proper $P^e$-primary ideal which gives that $(\Gamma_{I^e}^{R_P})^c$ is $P$-primary, and hence, $(\Gamma_{I^e}^{R_P})^c\neq P$. By \cite[Corollary.~3.1]{HRR2002}, we have that $((\Gamma_{I^e}^{R_P})^c:_RP)^e$ is a principal ideal. Notice that since $R$ is Noetherian, $P$ is finitely generated, and thus, \[((\Gamma_{I^e}^{R_p})^c:_RP)^e=\Gamma_{I^e}^{R_p}:_{R_p}P^e=IR_P,\] where the last equality follows by Theorem \ref{bijection}.
   \par (3) Follows from the equivalent statement in \cite[Corollary.~3.1]{HRR2002}.
\end{proof}

The above result gives the following extension of Theorem~\ref{bijection}.

\begin{corollary}
    Let $R$ be a Noetherian ring and $P$ a prime ideal of $R$. There exists a bijection between the $P$-primary strongly irreducible ideals which are not equal to $P$ and proper strongly hollow ideals of $R_P$.
\end{corollary}

\begin{proof}
    Consider a $P$-primary strongly irreducible ideal $I\neq P$. Suppose $(I:_RP)R_P\subseteq JR_P+KR_P$, for two ideals $JR_P$, $KR_P$ in $R_P$ with $(I:_RP)R_P\not\subseteq J$ and $(I:_RP)R_P\not \subseteq K$. By \cite[Corollary 3.1]{HRR2002}, we must have that $J^{ec}+K^{ec}\subseteq I$, which gives  \[I:_RP\subseteq ((I:_RP)R_p)\cap R\subseteq JR_P\cap R+KR_P\cap R\subseteq I.\] Then $I=(I:_RP)R_P\cap R$,
    which is impossible since in a Noetherian ring, every proper $P$-primary ideal $I$ satisfies
\[
(I :_{R} P)R_p\cap R\ \supsetneq\ I.
\]
Indeed, localize at $P$ and note $IR_P=IR_{P}$ and $PR_P=PR_{P}\subset R_{P}$. Because $I$ is $P$-primary,
$IR_P$ is $PR_P$-primary in $R_{P}$; hence there exists $n\ge1$ with
\[
(PR_P)^{n}\ \subseteq\ IR_P\ \subsetneq\ R_{P}\quad\text{and}\quad (PR_P)^{n-1}\ \nsubseteq\ IR_P.
\]
Choose $x\in (PR_P)^{n-1}\setminus IR_P$. Then
\[
PR_Px\ \subseteq\ PR_P\cdot(PR_P)^{n-1}\ =\ (PR_P)^{n}\ \subseteq\ IR_P,
\]
so $x\in(IR_P :_{R_{P}} PR_P)\setminus IR_P$, and therefore
\[
(IR_P :_{R_{P}} PR_P)\ \supsetneq\ IR_P.
\]
Since $P$ is finitely generated, colons commute with localization on the right, giving
\[
(I :_{R} P)R_P\ =\ (IR_P :_{R_{P}} PR_P)\ \supsetneq\ IR_P.
\]
Contracting yields
\[
(I :_{R} P)R_P\cap R\ \supsetneq\ IR_P\cap R\ =\ I,
\]
a contradiction.
\end{proof}

\begin{corollary}
    Let $R$ be a Noetherian ring. There exists a completely strongly hollow ideal in $R_P$, for some $P\in \mathrm{Spec}(R)$ if and only if there exists a non-prime strongly irreducible ideal in $R$. 
\end{corollary}
\begin{proposition}
    Let $I\subseteq J(R)$ be a strongly hollow ideal of a Noetherian ring $R$. Then $I=(\Gamma_I:M)$, for some maximal ideal $M$ if and only if $R$ is local.
\end{proposition}

\begin{proof}
    We have already shown the case for a local ring $R$ in Theorem~\ref{ci_gives_ch}. Suppose that $I=(\Gamma_I:M)$ for some maximal ideal $M$. Since $I\subseteq J(R)$, we have that $\Gamma_I$ is not a prime ideal. Furthermore, $\Gamma_I$ is $M$-primary, for some maximal ideal $M$ (see \cite[Corollary 1.5]{FHO}). Since $R$ is Noetherian, there exists $n\in\mathds{N}$ such that $M^n\subseteq \Gamma_I$. Therefore $M^{n-1}\subseteq (\Gamma_I:M)$. Suppose, by way of contradiction that $(\Gamma_I:M)\not\subseteq M$. Then $(\Gamma_I:M)R_M=R_M$ which gives the contradiction $\Gamma_IR_M=MR_M$. Therefore,  $\Gamma_I:M$ is $M$-primary. By assumption, $I$ is $M$-primary. Therefore $I\not\subseteq N$, for all $N\in\mathrm{Max}(R)\textbackslash\{M\}$. By assumption $I\subseteq J(R)$, so we conclude that $R$ is local. 
\end{proof}

\begin{proposition}
    Let $I\subseteq J(R)$ be a strongly hollow ideal of a Noetherian ring $R$. If $I\not\subseteq \mathrm{Nil}(R)$, then $\Gamma_I=M^n$, where $M=\Gamma_I:I$.
\end{proposition}

\begin{proof}
    If $\mathrm{Ht}(\Gamma_I)=0$, then every prime ideal lies above $I$, and hence $I\subseteq \mathrm{Nil}(R)$. Therefore, if $I\not\subseteq \mathrm{Nil}(R)$, we must have  $\mathrm{Ht}(\Gamma_I)>0$. Hence, $\Gamma_I=M^n$ (see \cite[Theorem.~3.6]{HRR2002}). To finish the proof, we notice that $\Gamma_I\subset \Gamma_I:I$, and the latter is a maximal ideal \cite[Proposition.~1.20]{Rostami2021}.
\end{proof}

\begin{lemma}\label{Art_max_csi}
    Every strongly irreducible ideal of an Artinian ring $R$ is completely strongly irreducible. 
\end{lemma}

\begin{proof}
    Let $K$ be a strongly irreducible ideal of an Artinian ring $R$. Then $K$ is completely irreducible (see \cite[page~1]{FHO}) and strongly irreducible.
   Let $\{C_\omega\}_{\omega \in \Omega}$ be a family of ideals of $R$ such that $\bigcap_{\omega\in\Omega}C_\omega\subseteq K$. Since $R$ is Artinian, we have that $\bigcap_{\omega\in\Omega}C_\alpha=\bigcap_{\omega\in F}C_\alpha$, for some finite subset $F\subset\Omega$. Since $K$ is strongly irreducible,  there exists $\omega \in \Omega $ such that $K\supseteq C_\omega$. 
\end{proof}

\begin{proposition}
    In  an Artinian ring $R$, every maximal ideal  corresponds to a minimal completely strongly hollow ideal.
\end{proposition}

\begin{proof}
    Let $M$ be a maximal ideal of $R$. Then $M$ is completely strongly irreducible by Lemma \ref{Art_max_csi}. Therefore, there exists a completely strongly hollow ideal $I$ which is the least ideal with respect to not being contained in $M$. By the order preserving property of the bijection between completely strongly irreducible and completely strongly hollow ideals and the maximality of $M$, we have that $I$ is minimal amongst completely strongly hollow ideals. 
\end{proof}

\smallskip
\section{Strongly Hollow ideals and Common Divisors}
\label{4}
We begin this final section by defining a new property of ideals in GCD rings. We denote the greatest common divisor (when it exists) of two elements $x$ and $y$ by $\mathrm{g.c.d.}(x,y)$. 

\begin{definition}
    We say that an ideal $I$ of a GCD ring $R$ satisfies ($\star$) if \[I\subseteq R(\mathrm{g.c.d.}(x,y))\implies I\subseteq Rx \text{ or } I\subseteq Ry \quad \forall x,y\in R.\]
\end{definition}

\begin{proposition}
    Let $I$ be an ideal of a GCD ring $R$. If $I$ is finitely generated and satisfies $(\star)$, then $I$ is strongly hollow. 
\end{proposition}

\begin{proof}
    Suppose $I\subseteq S+T$ but $I\not\subseteq S$ and $I\not\subseteq T$. Since $I$ is finitely generated, there exist two finite subsets, $F_1$ and $F_2$, of $S$ and $T$ respectively such that \[I\subseteq \sum_{s\in F_1}Rs+ \sum_{t\in F_2}Rt.\] Notice that for each pair of elements $s\in S$ and $t\in T$, we have that $I\not\subseteq R(\mathrm{g.c.d.}(x,y))$. Hence $I\not\subseteq Rs+Rt$. Now, fix $s_1\in F_1$ and notice that for all $t\in F_2$, we have $I\not\subseteq R(\mathrm{g.c.d.}(s_1,t))$. Therefore, pick $t_1\in F_2$ and set $x_1:=\mathrm{g.c.d.}(s_1,t_1).$ Then, pick $t_2\in F_2$ and notice that $I\subseteq R(\mathrm{g.c.d.}(x_1,t_1))$. Set $x_2:=\mathrm{g.c.d.}(x_1,t_2)$ and repeat this process until all elements of $F_2$ have been used. Then we have that $Rx_n$, where $n:=|F_2|$, is a principal ideal which contains $\sum_{t\in F_2}Rt$ and which does not contain $I$. Therefore $I\subseteq \sum_{s\in F_1}Rs+Rx_n$. Now repeat the same process with the elements of $F_1$ and construct a principal ideal $Ry_m$ such that it contains the first summand and does not contain $I$. Then we have that $I\subseteq Ry_m+Rx_n\subseteq R(\mathrm{g.c.d.}(y_m,x_n))$, which contradicts the fact that $I\not\subseteq Ry_m$ and $I\not\subseteq Rx_n$. Hence $I\subseteq S$ or $I\subseteq T$.
\end{proof}

\begin{proposition}
    Let $I$ be an ideal of a GCD ring $R$ which satisfies the ascending chain condition on principal ideals (ACCP). If $I$ satisfies $(\star)$, then $I$ is strongly hollow.
\end{proposition}

\begin{proof}
    Suppose $I\subseteq S+T$. Then consider the set of principal ideals generated by some chain of greatest common divisors of elements in $S$, \[U_S=\{Rx \;|\; x=\underbrace{\mathrm{g.c.d.}(...(\mathrm{g.c.d.}(\mathrm{g.c.d.}(\mathrm{g.c.d.}( }_{n\;\text{times}}k_1, k_2 ),k_3),k_4),\dots,k_n ;\quad\; k_i\in K \text{ for every } i\in[1,n]\}.\] Let $C$ be a chain in $U_S$. Since $C$ is a chain of principal ideals, it has a maximal element by ACCP and therefore, an upper bound. Then by Zorn's lemma, $U_S$ has a maximal element which we denote by $x_S$. For all $s\in S$, we have that $R(\mathrm{g.c.d.}(x_S,s))\in U$ and $R(\mathrm{g.c.d.}(x_S,s))\supseteq Rx_S$. Therefore, $R(\mathrm{g.c.d.}(x_S,s))=Rx_S$, and so $s\in Rx_S$. Hence $Rx_S\supseteq S$. Furthermore, since $x_S$ is a finite chain of greatest common divisors of elements in $S$, we have that $I\not\subseteq Rx_S$ since that would contradict $I\not\subseteq S$. Proceed with the same construction over $T$ and denote the resulting element by $x_T$. Then $I\subseteq Rx_S+Rx_T$, but $I\not\subseteq Rx_S$ and $I\not\subseteq Rx_T$, contradicting the assumption that $I$ satisfied $(\star)$. Hence, it must be the case that either $I\subseteq S$ or $I\subseteq T$.  
\end{proof}

\begin{remark}
The converse to the above result holds in any Bezout ring.
\end{remark}

\end{document}